\documentclass[a4j,12pt,draft]{article}
\usepackage{amsmath}
\usepackage{amssymb}
\usepackage{amsthm}
\usepackage[mathscr]{eucal}
\usepackage{tikz}
\usepackage[all,knot,cmtip]{xy}
\usepackage{color}
\xyoption{arc}
\entrymodifiers={+!!<0pt,\fontdimen22\textfont2>} 

\numberwithin{equation}{section}
\newtheorem{definition}{Definition}[section]

\newtheorem{lem}[definition]{Lemma}

\newtheorem{thm}[definition]{Theorem}

\newtheorem{cor}[definition]{Corollary}

\title{The non-nil-invariance of TP}
\author{Ryo Horiuchi}
\date{}
\begin{document}

\newpage
\maketitle


\section{Introduction}
In \cite{Hesselholt1}, Hesselholt defined a spectrum $\operatorname{TP}(X)$, the periodic topological cyclic homology of a scheme $X$, using topological Hochschild homology and the Tate construction, which is a topological analogue of Connes-Tsygan periodic cyclic homology $\operatorname{HP}$ defined by Hochschild homology and the Tate construction. In \cite[Theorem II.5.1]{Goodwillie}, Goodwillie proved that for $R$ an algebra over a field of characteristic 0 and $I$ a nilpotent ideal of $R$, the quotient map $R\to R/I$ induces an isomorphism on $\operatorname{HP}$. In this paper, we show that the analogous result for $\operatorname{TP}$ does not hold, that is to say, there is an algebra of positive characteristic and a nilpotent ideal such that the quotient map does not induce an isomorphism on $\operatorname{TP}$, even rationally. More precisely, we prove the following result.
\begin{thm}Let $p$ be a prime number and $k\geq2$ a natural number. Then the canonical map $$\operatorname{TP}_*(\mathbb{F}_p[x]/(x^k))\to \operatorname{TP}_*(\mathbb{F}_p)$$ is not an isomorphism. Moreover, if $k$ is not a $p$-power, then the map is also not an isomorphism after inverting $p$.
\end{thm}

In \cite{Hesselholt1}, Hesselholt gives a cohomological interpretation by TP of the Hasse-Weil zeta function of a scheme smooth and proper over a finite field inspired by \cite{Deninger} and \cite{CC}. Furthermore, in \cite{AMN} and \cite{BM1}, it is proved that $\operatorname{TP}$ satisfies the K\"unneth formula for smooth and proper dg-categories over a perfect field of positive characteristic. Therefore, this new cohomology theory TP may be considered to be an important cohomology theory for $p$-adic geometry and non-commutative geometry. Our result concerns a fundamental property of this theory. In Theorem 3.3, we evaluate the TP-group of $\mathbb{F}_p[x]/(x^k)$ completely.

\section{Periodic topological cyclic homology}

Periodic topological cyclic homology $\operatorname{TP}$ is proposed in \cite{Hesselholt1}. In this section, we briefly recall some notions from there. We let $\mathbb T$ denote the circle group. The following construction written in the higher categorical language can be found at \cite[I.4]{NS}.

Let $E$ be a free $\mathbb T$-CW-complex whose underlying space is contractible. Then we consider the following cofibration sequence of pointed $\mathbb{T}$-spaces
\[E_+\to S^0\to \tilde{E},\]
where $E_+$ is the pointed space $E\sqcup \{\infty\}$ and $S^0=\{0, \infty\}$, and the left map sends $\infty$ to the base point $\infty\in S^0$ and all other points to $0\in S^0$.

Let $X$ be a $\mathbb T$-spectrum. Smashing the internal hom spectrum $[E_+, X]$ with the above diagram and taking fixed points of a subgroup $C\subset\mathbb T$, we have the following sequence called Tate cofibration sequence
\[(E_+\otimes_{\mathbb{S}}[E_+, X])^C\to ([E_+, X])^C\to (\tilde{E}\otimes_{\mathbb{S}}[E_+, X])^C.\]

\hspace{-6mm}We write this sequence as
\begin{eqnarray}\nonumber
(E_+\otimes[E_+, X])^C=
\begin{cases}
   \operatorname{H}_\cdot(C, X), & if\ C\subsetneq\mathbb{T}  \\
    \Sigma{\rm H}_{\cdot}(C, X), & if\ C=\mathbb{T},
\end{cases}\end{eqnarray}
\vspace{-7mm}\begin{eqnarray}
([E_+, X])^C={\rm H}^{\cdot}(C, X) \nonumber\\
(\tilde{E}\otimes[E_+, X])^C=\hat{\rm H}^{\cdot}(C, X)\nonumber
\end{eqnarray}


Let $X$ be a scheme. The topological periodic cyclic homology of $X$ is the spectrum given by $$\operatorname{TP}(X)=\hat{\rm H}^{\cdot}(\mathbb{T}, {\rm THH}(X)),$$ where ${\rm THH}$ denotes the topological Hochschild homology of $X$ defined in \cite{GH} and \cite{BM2}.
In the present paper, we will only consider affine schemes. For a commutative ring $R$, there is a conditionally convergent spectral sequence \cite[$\S$4]{HM3},
\[E^2_{i,j}=S\{t, t^{-1}\}\otimes{\rm THH}_j(R)\Rightarrow \operatorname{TP}_{i+j}(R),\]
where deg$(t)=(-2, 0)$.



\section{Truncated polynomial algebras}
Our main result is the following

\begin{thm}Let $p$ be a prime number and $k\geq2$ a natural number. If $k$ is not a $p$-power, then the canonical map $$\operatorname{TP}_*(\mathbb{F}_p[x]/(x^k))[1/p]\to \operatorname{TP}_*(\mathbb{F}_p)[1/p]$$ is not an isomorphism.
\end{thm}

Before proving our main result,  we recall from \cite{HM2} and \cite{Hesselholt2} some calculations concerning ${\rm THH}(\mathbb{F}_p[x]/(x^k))$. The following also shown in \cite[Paper B]{S} in the higher categorical language.

We give the pointed finite set $\Pi_k$=$\{0, 1, x, \dots, x^{k-1}\}$ with the base point $0$ the pointed commutative monoid structure, where $1$ is the unit, $0\cdot1=0\cdot x^i=0$, $x^i\cdot x^j=x^{i+j}$ and $x^k=0$. We denote the cyclic bar construction of $\Pi_k$ by ${\rm N}^{\operatorname{cy}}_{\bullet}(\Pi_k)$. More precisely, the set of $l$-simplicies is

$${\rm N}^{\operatorname{cy}}_l(\Pi_k)=\Pi_k\wedge\dots\wedge\Pi_k,$$
where there are $l+1$ smash factors and the structure maps are given by

$d_i(x_0\wedge\dots\wedge x_l)=x_0\wedge\dots\wedge x_ix_{i+1}\wedge\dots\wedge x_l, \ 0\leq i < l,$\

$d_l(x_0\wedge\dots\wedge x_l)=x_lx_0\wedge x_1\wedge\dots\wedge x_{k-1},$\

$s_i(x_0\wedge\dots\wedge x_l)=x_0\wedge\dots\wedge x_i\wedge 1 \wedge x_{i+1}\wedge\dots\wedge x_l, \ 0\leq i \leq l,$\

$t_l(x_0\wedge\dots\wedge x_l)=x_l\wedge x_0\wedge x_1\wedge\dots\wedge x_{l-1}$.
\ \\
We let ${\rm N}^{\operatorname{cy}}(\Pi_k)$ denote the geometric realization of ${\rm N}^{\operatorname{cy}}_{\bullet}(\Pi_k)$.

In \cite[Theorem 7.1]{HM1}, it is proved that there is a natural equivalence of cyclotomic spectra
\begin{equation} {\rm THH}(\mathbb{F}_p[x]/(x^k))) \simeq {\rm THH}(\mathbb{F}_p)\otimes {\rm N}^{\operatorname{cy}}(\Pi_k). \tag{a}
\end{equation}

\hspace{-6.5mm}For each positive integer $i$, we also have the cyclic subset $${\rm N}^{\operatorname{cy}}_{\bullet}(\Pi_k, i)\subset{\rm N}^{\operatorname{cy}}_{\bullet}(\Pi_k)$$ generated by the $(i-1)$-simplex $x\wedge\dots\wedge x$ ($i$ factors), and denote the geometric realization by ${\rm N}^{\operatorname{cy}}(\Pi_k, i)$. We also have the cyclic subset ${\rm N}^{\operatorname{cy}}_{\bullet}(\Pi_k, 0)$ generated by the $0$-simplex 1 with the geometric realization ${\rm N}^{\operatorname{cy}}(\Pi_k, 0)$. Thus we obtain the following wedge decomposition
\[\bigvee_{i\geq0}{\rm N}^{\operatorname{cy}}(\Pi_k, i)={\rm N}^{\operatorname{cy}}(\Pi_k).\]

We consider the complex $\mathbb T$-representation, where $d=\lfloor(i-1)/k\rfloor$ is the integer part of $(i-1)/k$ for $i\geq1$,

\[\lambda_d=\mathbb{C}(1)\oplus\mathbb{C}(2)\oplus\dots\oplus\mathbb{C}(d),\] where $\mathbb{C}(i)=\mathbb{C}$ with the $\mathbb{T}$ action; $$\mathbb{T}\times\mathbb{C}(i)\to\mathbb{C}(i)$$ defined by $\ (z,w)\mapsto z^iw.$ Then we have the following by \cite[theorem B]{HM2}, for $i\geq1$ such that $i\notin k\mathbb{N}$, there is an equivalence
\begin{equation} {\rm N}^{\operatorname{cy}}(\Pi_k, i)\simeq S^{\lambda_d}\wedge(\mathbb{T}/C_{i})_+, \tag{b} \end{equation}
where $C_{i}$ is the $i$-th cyclic group.

Let $\operatorname{THH}(\mathbb{F}_p[x]/(x^k), (x))$ denote the fiber of the canonical map $$\operatorname{THH}(\mathbb{F}_p[x]/(x^k))\to\operatorname{THH}(\mathbb{F}_p),$$ and we write 
$$\operatorname{TP}(\mathbb{F}_p[x]/(x^k), (x))=\hat{\rm H}^{\cdot}(\mathbb{T}, \operatorname{THH}(\mathbb{F}_p[x]/(x^k), (x))).$$

The non-triviality of $\operatorname{TP}(\mathbb{F}_p[x]/(x^k), (x))$ shall imply that $\operatorname{TP}$ is not nil-invariant. In order to obtain the non-triviality, we use the following decomposition.
\begin{lem}
There is a canonical equivalence
\[\operatorname{TP}(\mathbb{F}_p[x]/(x^k), (x))\simeq \prod_{i\geq1} \hat{\rm H}^{\cdot}(\mathbb{T}, {\rm THH}(\mathbb{F}_p)\otimes {\rm N}^{cy}(\Pi_k, i)).\]
\end{lem}

\begin{proof}
By (a) and the wedge decomposition of ${\rm N}^{\operatorname{cy}}(\Pi_k)$ above, we have
\[\Sigma{\rm H}_{\cdot}(\mathbb{T}, {\rm THH}(\mathbb{F}_p[x]/(x^k), (x)))\simeq\bigoplus_{i\geq1}\Sigma{\rm H}_{\cdot}(\mathbb{T}, {\rm THH}(\mathbb{F}_p)\otimes{\rm N}^{cy}(\Pi_k, i)),\]
since $\Sigma{\rm H}_{\cdot}(\mathbb{T}, -)$ preserves all homotopy colimits.

\hspace{-6mm}Since the connectivity of $\Sigma{\rm H}_{\cdot}(\mathbb{T}, {\rm THH}(\mathbb{F}_p)\otimes{\rm N}^{cy}(\Pi_k, i))$ goes to  $\infty$ as $i$ goes to $\infty$, we have
 \[\bigoplus_{i\geq1}\Sigma{\rm H}_{\cdot}(\mathbb{T}, {\rm THH}(\mathbb{F}_p)\otimes{\rm N}^{cy}(\Pi_k, i))\simeq\prod_{i\geq1}\Sigma{\rm H}_{\cdot}(\mathbb{T}, {\rm THH}(\mathbb{F}_p)\otimes{\rm N}^{cy}(\Pi_k, i)).\]
Similarly, since $\operatorname{H}^{\cdot}(\mathbb{T}, -)$ preserves all homotopy limits, we have 
\[{\rm H}^{\cdot}(\mathbb{T}, {\rm THH}(\mathbb{F}_p[x]/(x^k), (x)))\simeq\prod_{i\geq1}{\rm H}^{\cdot}(\mathbb{T}, {\rm THH}(\mathbb{F}_p)\otimes{\rm N}^{cy}(\Pi_k, i)).\]
Lastly, since $\operatorname{TP}(\mathbb{F}_p[x]/(x^k), (x))$ is the cofiber of the map $$\Sigma{\rm H}_{\cdot}(\mathbb{T}, {\rm THH}(\mathbb{F}_p[x]/(x^k), (x)))\to{\rm H}^{\cdot}(\mathbb{T}, {\rm THH}(\mathbb{F}_p[x]/(x^k), (x))),$$ we get the desired equivalence.
\end{proof}


It is known that, for a $\mathbb{T}$-spectrum $X$, there is a $\mathbb{T}$-equivalence \[X\otimes(\mathbb{T}/C_{i})_+\simeq\Sigma[(\mathbb{T}/C_{i})_+, X],\] see for example \cite[8.1]{HM1}. Hence, we have
\begin{eqnarray*}
\hat{\rm H}^{\cdot}(\mathbb{T}, {\rm THH}(\mathbb{F}_p)\otimes({\mathbb T}/C_{i})_+) & = &(\tilde{E}\otimes[E_+, {\rm THH}(\mathbb{F}_p)\otimes({\mathbb T}/C_{i})_+])^{\mathbb T} \nonumber \\
&\simeq & \Sigma(\tilde{E}\otimes[E_+, [(\mathbb{T}/C_{i})_+, {\rm THH}(\mathbb{F}_p)]])^{\mathbb{T}}\nonumber\\
&\simeq&\Sigma(\tilde{E}\otimes[(\mathbb{T}/C_{i})_+, [E_+, {\rm THH}(\mathbb{F}_p)]])^{\mathbb{T}} \\
&\simeq&(\tilde{E}\otimes(\mathbb{T}/C_{i})_+\otimes[E_+, {\rm THH}(\mathbb{F}_p)]])^{\mathbb{T}} \\
&\simeq&\Sigma([(\mathbb{T}/C_{i})_+, \tilde{E}\otimes[E_+, {\rm THH}(\mathbb{F}_p)]])^{\mathbb{T}} \\
& \simeq & \Sigma(\tilde{E}\otimes_{\mathbb{S}}[E_+, {\rm THH}(\mathbb{F}_p)])^{C_{i}} \nonumber \\
& = & \Sigma\hat{\rm H}^{\cdot}(C_{i}, {\rm THH}(\mathbb{F}_p)). \nonumber
\end{eqnarray*}
Furthermore, we have an equivalence of spectra
\[\hat{\rm H}^{\cdot}(C_{i}, {\rm THH}(\mathbb{F}_p)\otimes S^{\lambda_d})\simeq\hat{\rm H}^{\cdot}(C_{p^{v_p(i)}}, {\rm THH}(\mathbb{F}_p)\otimes S^{\lambda_d}),\]
where $v_p$ denotes the $p$-adic valuation.

Hesselholt and Madsen have calculated the homotopy groups of the above spectra \cite[$\S9$]{HM1},
$$\pi_*\hat{\rm H}^{\cdot}(C_{p^n}, {\rm THH}(\mathbb{F}_p)\otimes S^{\lambda_d})\cong S_{{\mathbb{Z}/p^n\mathbb{Z}}}\{t, t^{-1}\},$$
where $t$ is the divided Bott element. More precisely, $\pi_*\hat{\rm H}^{\cdot}(C_{p^n}, {\rm THH}(\mathbb{F}_p)\otimes S^{\lambda_d})$ is a free module of rank $1$ over $\mathbb{Z}/p^n\mathbb{Z}[t, t^{-1}]$ on a generator of degree $2d$. A preferred generator is specified in \cite[Proposition 2.5]{Hesselholt3}. Combining these and (b), we obtain for $i\notin k\mathbb{N}$ a canonical isomorphism
\[
\pi_j\hat{\rm H}^{\cdot}(\mathbb{T}, {\rm THH}(\mathbb{F}_p)\otimes\operatorname{N^{cy}}(\Pi_k, i)) \cong 
\begin{cases}
    \mathbb{Z}/p^{v_p(i)}\mathbb{Z}, & j-\lambda_d+1 \ \text{even}  \\
    0, & j-\lambda_d+1 \ \text{odd},
  \end{cases}
\]
and note that $-\lambda_d+1$ is always odd by definition. They have similarly showed that for $i\in k\mathbb{N}$, there is a canonical isomorphism
\[
\pi_j\hat{\rm H}^{\cdot}(\mathbb{T}, {\rm THH}(\mathbb{F}_p)\otimes\operatorname{N^{cy}}(\Pi_k, i)) \cong 
\begin{cases}
    \mathbb{Z}/p^{v_p(k)}\mathbb{Z}, & j \ \text{odd}  \\
    0, & j \ \text{even}.
  \end{cases}
\]

From these, we obtain the following.
\begin{thm}\label{val} If $j$ is an odd integer, then there is a canonical isomorphism 
\begin{eqnarray*}
\operatorname{TP}_j(\mathbb{F}_p[x]/(x^k), (x))&\cong&\prod_{i\geq 1, i\in k\mathbb{N}}\mathbb{Z}/p^{v_p(k)}\mathbb{Z}\ \times\prod_{i\geq 1, i\notin k\mathbb{N}}\mathbb{Z}/p^{v_p(i)}\mathbb{Z}.
\end{eqnarray*}
If $j$ is an even integer, then $$\operatorname{TP}_j(\mathbb{F}_p[x]/(x^k), (x))=0.$$
\end{thm}

Therefore, we get our main result by this theorem. Moreover, this concrete calculation gives us the following.
\begin{cor}If $k=p^r$ with a natural number $r\in\mathbb{N}$, the canonical map
$$\operatorname{TP}_*(\mathbb{F}_p[x]/(x^{p^r}))[1/p]\to \operatorname{TP}_*(\mathbb{F}_p)[1/p]$$ is an isomorphism.
\end{cor}
Thus, in this specific case, the analogue of Goodwillie's theorem for TP holds. In addition, by \cite[Corollary 1.5]{NS} and the main theorem of  \cite{HM2}, we get the following.
\begin{cor}Topological negative cyclic homology is not nil-invariant, even rationally.
\end{cor}






\section*{Acknowledgements} I would like to thank Lars Hesselholt for his tremendous help and suggesting me this topic. I also thank Martin Speirs for reading the draft carefully and daily conversation and the DNRF Niels Bohr Professorship of Lars Hesselholt for the support.

\end{document}